\documentclass[11pt,a4paper]{amsart}
\usepackage{amsthm}
\usepackage{amsfonts}
\usepackage{amsmath}
\usepackage{amssymb}
\usepackage{dsfont}
\usepackage{txfonts}
\usepackage{cancel}
\usepackage{soul}
\usepackage{mathrsfs}
\usepackage{enumerate}
\usepackage{textcomp}
\usepackage{hyperref} 
\usepackage{color}

\author[C. Goel, S. Kuhlmann, B. Reznick]{Charu Goel, Salma Kuhlmann, Bruce Reznick}
\address{Department of Mathematics and Statistics, University of
Konstanz, Universit\"{a}tsstrasse 10, 78457 Konstanz, Germany}
\email{charu.goel@uni-konstanz.de}

\address{Department of Mathematics and Statistics, University of
Konstanz, Universit\"{a}tsstrasse 10, 78457 Konstanz, Germany}
\email{salma.kuhlmann@uni-konstanz.de}

\address{Department of Mathematics, University of
Illinois at Urbana-Champaign, Urbana, IL 61801}
\email{reznick@math.uiuc.edu}
 
\title[On the Choi-Lam analogue of Hilbert's 1888 theorem for Symmetric Forms]
{On the Choi-Lam analogue of Hilbert's 1888 theorem for Symmetric Forms\ }
\date{}

\keywords{Positive Polynomials, Sums of Squares, Symmetric forms}
\subjclass[2010]{11E76, 11E25, 05E05}
 
\begin{document}
\maketitle

\theoremstyle{definition}
  \newtheorem{thm}{Theorem}[section]
 \newtheorem{df}[thm]{Definition}%[section]
 \newtheorem{nt}[thm]{Notation}%[section]
 \newtheorem{prop}[thm]{Proposition}%[section]
 \newtheorem{rk}[thm]{Remark}%[section]
 \newtheorem{lem}[thm]{Lemma}
 \newtheorem{obs}[thm]{Observation}
 \newtheorem{cor}[thm]{Corollary}
 \newtheorem{nthng}[thm]{}

\begin{abstract} 
A famous theorem of Hilbert from 1888 states that a positive semidefinite (psd) real form is a sum of squares (sos) of real forms if and only if $n=2$ or $d=1$ or $(n,2d)=(3,4)$, where $n$ is the number of variables and $2d$ the degree of the form. In 1976, Choi and Lam proved the analogue of Hilbert's Theorem for symmetric forms by assuming the existence of psd not sos symmetric $n$-ary quartics for $n \geq 5$. In this paper we complete their proof by constructing explicit psd not sos symmetric $n$-ary quartics for $n \geq 5$. 
\end{abstract}

\section{Introduction}
A real form (homogeneous polynomial) $f$ is called \textit{positive semidefinite} (psd) if it takes only non-negative values and it is called a \textit{sum of squares} (sos) if there exist other forms $h_j$ so that $p = h_1^2 + \cdots + h_k^2$. The question whether a real psd form can be written as a sum of squares of real forms has many ramifications and has been studied extensively. Since a psd form always has even degree, it is sufficient to consider this question for even degree forms. We refer to this question as $(\mathcal{Q})$.  The first significant result in this direction was given by D. Hilbert \cite{Hilb_1} in 1888. 
His celebrated theorem states that a psd form is sos if and only if  $n=2$ or $d=1$ or $(n,2d)=(3,4)$, where $n$ is the number of variables and $2d$ the degree of the form.

The above answer to ($\mathcal{Q}$)  
can be summarized by  the following chart:
\vspace{0.2cm}

\begin{center}
    \begin{tabular}{ | l | l | l | l | l | l | p{0.4cm} |}
    \hline
   deg $\setminus$ var & 2 & 3 & 4 & 5 & 6 & $\ldots$ \\ \hline
    2 & $\checkmark$  & $\checkmark$ & $\checkmark$ & $\checkmark$ & $\checkmark$ & $\ldots$ \\ \hline
    4& $\checkmark$ & $\checkmark$ & $\times$ & $\times$  & $\times$ & $\ldots$  \\ \hline
    6 & $\checkmark$ & $\times$& $\times$ &$\times$ & $\times$ & $\ldots$  \\ \hline
      8 & $\checkmark$ & $\times$& $\times$& $\times$ & $\times$ & $\ldots$ \\ \hline
        $\vdots$ &  $\vdots$ & $\vdots$ & $\vdots$ & $\vdots$ & $\vdots$ & $\ddots$\\ \hline
    \end{tabular}
\end{center}
\vspace{0.2cm}

 \noindent where, a tick ($\checkmark$) denotes a positive answer to ($\mathcal{Q}$), whereas a cross ($\times$)  denotes a negative  answer to ($\mathcal{Q})$. 

Let $\mathcal{P}_{n,2d}$ and $\Sigma_{n,2d}$ denote the cone of psd and sos $n$-ary $2d$-ic forms (i.e. forms of degree $2d$ in $n$ variables) respectively. Hilbert made a careful study of quaternary quartics and ternary sextics, and demonstrated that $\Sigma_{3,6}\subsetneq \mathcal{P}_{3,6}$ and $\Sigma_{4,4}\subsetneq\mathcal{P}_{4,4}$.  
Moreover he showed that 

\vspace{-0.3cm}

\begin{equation}  \text{if} \ \Sigma_{4,4} \subsetneq \mathcal{P}_{4,4} \  \text{and} \ \Sigma_{3,6}\subsetneq \mathcal{P}_{3,6}, \ \text{then} 
\nonumber \end{equation}
\begin{equation} \label{eq: basic cases inequality} \Sigma_{n,2d}\subsetneq\mathcal{P}_{n,2d} \ \text{for all} \ n \geq 3, 2d \geq 4 \ \text{and} \ (n,2d) \neq (3,4). \end{equation}

\noindent So it is sufficient to produce psd not sos forms in these two crucial cases of quaternary quartics and ternary sextics to get psd not sos forms in all remaining cases as in assertion (\ref{eq: basic cases inequality}) above.
\noindent In those two cases Hilbert described a method to produce examples of psd not sos forms, which was \textquotedblleft elaborate and unpractical\textquotedblright \ (see \cite[p.387] {CL_1}), so no explicit examples appeared in literature for next 80 years. 
Explicit examples with $(n,2d)=(3,6)$ were found by T. S. Motzkin \cite{Motz-1} in 1967 and R. M. Robinson \cite{Rob-1} in 1969; Robinson also found an explicit example with $(n,2d)=(4,4)$. M. D. Choi and T. Y. Lam \cite{Choi, CL_1, CL-2} produced many more examples in the mid 1970's. 
More examples were given later by B. Reznick \cite{Rez} and K. Schm\"udgen \cite{Schmuedgen}.

In 1976, Choi and Lam \cite{CL_1} considered the question when a psd form is sos for the special case when the form considered is moreover symmetric ($f(x_{\sigma(1)},\ldots,x_{\sigma(n)})= f(x_1,\ldots,x_n)$ $\forall \ \sigma \in S_n$). Let $S\mathcal{P}_{n,2d}$ and $S\Sigma_{n,2d}$ denote the set of symmetric psd and symmetric sos $n$-ary $2d$-ic forms respectively. 
They demonstrated that it is enough to find symmetric psd not sos forms in the two crucial cases of $n$-ary quartics for $n \geq 4$ and ternary sextics to obtain symmetric psd not sos $n$-ary $2d$-ics for all $n \geq 3, 2d \geq 4$ and $(n,2d) \neq (3,4)$ (see Proposition \ref{prop:(n,4),(3,6)Not equal Implies All}). They showed that the answer to this question is the same as the answer to $(\mathcal{Q})$, by assuming the existence of psd not sos symmetric $n$-ary quartics for $n \geq 5$.  
For the convenience of the reader we include the following citation from \cite{CL_1}, 

\begin{quotation}
\textit{\textquotedblleft the construction of $f_{n,4} \in S\mathcal{P}_{n,4} \setminus S\Sigma_{n,4} \ (n \geq 4)$ requires considerable effort, so we shall not go into the full details here. Suffice it to record the special form $f_{4,4}=\sum x^2 y^2 + \sum x^2 y z - 2xyzw$. Here the two summations denote the full symmetric sums (w.r.t. the variables $x,y,z,w$); hence the summation lengths are respectively 6 and 12\textquotedblright.} 
\end{quotation}

\noindent We complete their proof by constructing explicit psd not sos symmetric $n$-ary quartics for $n \geq 5$ (see Theorems \ref{ODDvarPSDnotSOSquartic} and \ref{EVENvarPSDnotSOSquartic}). 
These theorems will be further used in \cite{G-K-R-2}, where we consider the question when an even symmetric  psd form is sos.

\section{Analogue of Hilbert's 1888 Theorem for Symmetric forms}
\label{sec:psdNOTsosSymmQuartics}

We revisit the question: for which pairs $(n,2d)$ will a symmetric psd $n$-ary $2d$-ic form be sos?  
We refer to this question as $\mathcal{Q}(S)$.

Choi and Lam in  \cite{CL_1} claimed that the answer to $\mathcal{Q}(S)$, that classifies the pairs $(n, 2d)$ for which a symmetric psd $n$-ary $2d$-ic form is sos, is: 
 \vspace{-0.4cm}

\begin{equation} \label{eq:symm}
S\mathcal{P}_{n,2d}= S\Sigma_{n,2d} \ \text{if and only if} \ n=2 \ \text{or} \  d=1 \ \text{or} \ (n,2d)=(3,4).
\end{equation}
  
One direction of (\ref{eq:symm}) follows from Hilbert's Theorem. Conversely for proving $S\mathcal{P}_{n,2d} \subseteq S\Sigma_{n,2d}$ only if $n=2$ or $d=1$ or $(n,2d) = (3,4)$, they showed that it is enough to find $f \in S\mathcal{P}_{n,2d} \setminus S\Sigma_{n,2d}$ for all pairs $(n,4)$ with $n \geq 4$ and for the pair $(3,6)$, i.e. 
they demonstrated that 
\vspace{-0.3cm}

\begin{equation} \text{if} \ S\Sigma_{n,4}\subsetneq S\mathcal{P}_{n,4} \ \text{for all} \ n \geq 4 \  \text{and} \ S\Sigma_{3,6}\subsetneq S\mathcal{P}_{3,6}, \ \text{then} 
\nonumber \end{equation} 
\begin{equation}  \label{eq: basic cases inequality Symm} S\Sigma_{n,2d} \subsetneq S\mathcal{P}_{n,2d} \ \text{for all} \ n \geq 3, 2d \geq 4 \ \text{and} \ (n,2d) \neq (3,4).\end{equation}

\begin{lem}  \label{irred indef preserves non sosness} Let $f \in \mathcal{F}_{n,2d}$ be a psd not sos form and $p$ an irreducible indefinite form of degree $r$ in $\mathbb{R}[x_1, \ldots, x_n]$. Then $p^2f \in \mathcal{F}_{n,2d+2r}$ is also a psd not sos form. 
\end{lem}
\begin{proof} Clearly  $p^2 f$ is psd. 
\noindent  If $p^2 f= \displaystyle \sum_{k} h_k^2$, then for every real tuple $\underline{a}$ with $p(\underline{a})=0$, it follows that $(p^2 f)(\underline{a})=0$.
\noindent This implies $h_k^2 (\underline{a})=0 \ \forall \ k$ (since $h_k^2$ is psd),
\noindent and so on the real variety $p=0$, we have $h_k=0$ as well.

So (using \cite[Theorem 4.5.1]{BCR}), for each $k$, there exists $g_k$ so that $h_k=p g_k$. This gives $f=\displaystyle \sum_{k} g_k^2$, which is a contradiction.
\end{proof}

\begin{prop} \label{prop:(n,4),(3,6)Not equal Implies All} If $S\Sigma_{n,4}\subsetneq S\mathcal{P}_{n,4}$ for all $n \geq 4$ and $S\Sigma_{3,6}\subsetneq S\mathcal{P}_{3,6}$, then 
$S\Sigma_{n,2d}\subsetneq S\mathcal{P}_{n,2d}$ for all $n \geq 3, d \geq 2$ and $(n,2d) \neq (3,4)$.
\end{prop}
\begin{proof}   Suppose we have forms $f \in S\mathcal{P}_{n,2d} \setminus S\Sigma_{n,2d}$ for all pairs $(n,4)$ with $n \geq 4$, and for the pair $(3,6)$. Then we can construct symmetric $n$-ary forms of higher degree by taking $(x_{1}+ \ldots + x_n)^{2i}f$, which  can be seen to be in \ $ S\mathcal{P}_{n,2d+2i} \setminus S\Sigma_{n,2d+2i} \ \forall \ i \ \geq 0$, by $i$ applications of Lemma  \ref{irred indef preserves non sosness} with $p=x_{1}+ \ldots + x_n$.
\end{proof}

For the pair $(3,6)$, Robinson  \cite{Rob-1} constructed the symmetric ternary sextic form $R(x,y,z):= x^6+y^6+z^6-(x^4y^2+y^4z^2+z^4x^2+x^2y^4+y^2z^4+z^2x^4)+3x^2y^2z^2$ and showed that it is psd but not sos. For the pair $(4,4)$, Choi and Lam \cite{CL_1} gave the form $f_{4,4} \in S\mathcal{P}_{4,4} \setminus S\Sigma_{4,4}$. So in view of Proposition \ref{prop:(n,4),(3,6)Not equal Implies All}, it remains to find psd not sos symmetric $n$-ary quartics for $n \geq 5$. 

\vspace{0.2cm}

We will now construct explicit forms $f \in S\mathcal{P}_{n,4} \setminus S\Sigma_{n,4}$ for $n \geq 4$. 
For $n \geq 4$, consider the symmetric $n$-ary quartic (studied in \cite{CLR-2})

$$L_n(\underline{x}):=\displaystyle m(n-m)\sum_{i<j}^{}(x_i-x_j)^4 - \Big(\sum_{i<j}^{}(x_i-x_j)^2\Big)^2,$$
where $m=\left \lfloor \frac{n}{2} \right \rfloor.$ We shall show that $L_n$ is psd for all $n$ and $L_n$ is not sos for all odd $n \geq 5$. 

We need an important result (Theorem \ref{atmost 2 distinct comps} below) of Choi, Lam and Reznick \cite{CLR-2}. The same argument was modified in \cite[Theorem 2.3]{Har} to treat even symmetric $n$-ary octics for $n \geq 4$.

\begin{thm} \label{atmost 2 distinct comps} A symmetric $n$-ary quartic $f$ is psd iff $f(\underline{x}) \geq 0$ for every $\underline{x} \in \mathbb{R}^n$ with at most two distinct coordinates (if $n\geq4$), i.e. $\Lambda_{n,2}=\{ \underline{x} \in \mathbb{R}^n \ | \ x_i \in \{r,s\}; r \neq s \}$ is a test set for symmetric $n$-ary quartics. \end{thm}

\begin{proof} See \cite[Corollary 3.11]{Goel}. 
\end{proof}

\begin{rk} V. Timofte's half degree principle \cite{Timofte-1} gives a complete generalisation of above theorem for both symmetric polynomials (i.e. invariant under the action of the group $S_n$) and even symmetric polynomials (i.e. invariant under the action of the group $S_n \times \mathbb{Z}_{2}^n$) of degree $2d$ in $n$ variables. 
See \cite{Kuhlmann-Kovacec-Riener} for an application of this principle to elementary symmetric functions. 
\end{rk}

\noindent 
For $n=5$, $L_n(\underline{x})$ has been discussed %in 1978 
by A. Lax and P. D. Lax. % in \cite{Lax}. 
They showed \cite[p.72]{Lax} that 
\[
A_5(\underline{x}):= \displaystyle \sum_{i=1}^{5} \prod_{j \neq i}^{} (x_i -x_j) = \frac{1}{8}L_5,
\]
 a psd symmetric quartic in five variables, is not sos.

\begin{prop} \label{psdnaryQuartics}
$L_n$ is psd for all $n$. 
\end{prop}
\begin{proof} In view of Theorem 
\ref{atmost 2 distinct comps}, it is enough to prove that $L_n \geq 0$ on the test set $\Lambda_{n,2}=\{(\underbrace{r, \ldots, r}_{k}, \underbrace{s, \ldots, s}_{n-k}) \ | \ r \neq s \in \mathbb{R}$; $0 \leq k \leq n \}$. 

Now for $\underline{x} \in \Lambda_{n,2}$, 
\[
x_i-x_j= \begin{cases} \pm (r-s) \neq 0 , \ \text{for} \ k(n-k) \ \text{terms}, \\ 0 \ \ \ \ \ \ \ \ \ \ \ \ \  \ \ \ \ \ , \ \text{otherwise} \end{cases}
\]

\vspace{0.2cm}

\noindent so $L_n$ takes the value 

\[
 \begin{gathered}
 L_n(\underline{x})= m(n-m) k(n-k)(r-s)^4-[k(n-k)(r-s)^2]^2 \\
%=k(n-k)(r-s)^4[m(n-m)-k(n-k)] \\
=k(n-k)(r-s)^4[(m-k)(n-m-k)], 
\end{gathered}
\]
which is non-negative since there is no integer between $m$ and $n-m$.
\end{proof}

\begin{df} \index{set!$0/1$}
\index{point!$0/1$}
Let $\{0,1\}^n$ be the set of all $n$-tuples $\underline{x} = (x_1, \ldots, x_n)$ with $x_i \in \{0,1\}$ for all $i = 1, \ldots, n$. A subset $S \subset \{0,1\}^n$ is called a \textbf{0/1 set} and $\underline{x} \in \{0,1\}^n$ a \textbf{0/1 point}. 
\end{df}

\begin{lem} \label{QUAD h identically 0} Suppose $n \geq 4$ and $h(x_1, \ldots, x_n)$ is a quadratic form that vanishes on all $0/1$ points with $m$ or $(m+1)$ $1$'s, where $m =\left \lfloor \frac{n}{2} \right \rfloor$, i.e. $h(\underline{x})=0$ for all $\underline{x}$ with $m$ or $(m+1)$ $1$'s and $\begin{cases} (m+1) \ \text{or} \ m \ 0 \text{'s}  \ \text{(respectively) for odd} \ n=2m+1;  \\ m \ \text{or} \ (m-1) \ 0  \text{'s}  \ \text{(respectively
) for even} \ n=2m. \end{cases}$

\noindent  Then $h$ is identically zero. 
\end{lem}
\begin{proof} Set $h(x_1, \ldots, x_n)=\displaystyle \sum_{i=1}^{n} a_i x_i^2 + \sum_{i < j}^{} a_{ij} x_i x_j$. Fix distinct $i, j, k$ and let $S$ such that $|S|=m-1$, be a set of indices not containing $i, j, k$.
\noindent Then $h=0$ on $\underline{x}$, where the $1$'s on $\underline{x}$ occur precisely on $S\cup \{i\}$, $S\cup \{i, k\}$, $S\cup \{j\}$, $S\cup \{j, k\}$. So we have:

\begin{equation} \hspace{-3.6cm} \text{on} \ S\cup \{i\}: \ \ \ \ ~ 0= \displaystyle \sum_{l \in S}^{} a_l + a_i + \sum_{l<l' \in S}^{} a_{ll'} + \sum_{l \in S}^{} a_{il}, \nonumber \end{equation} 

\begin{equation}  \hspace{-0.6cm} \text{on} \ S\cup \{i,k\}: \ 0= \displaystyle \sum_{l \in S}^{} a_l + a_i + a_k + \sum_{l<l' \in S}^{} a_{ll'} + \sum_{l \in S}^{} a_{il}+\sum_{l \in S}^{} a_{kl} +a_{ik}. \nonumber \end{equation}

\noindent Subtracting above two equations gives:
\begin{equation} \label{i,k} \displaystyle a_k + \sum_{l \in S}^{} a_{kl} +a_{ik} =0. \end{equation}

\noindent Doing the same with $S\cup \{j\}$ and $S\cup \{j, k\}$ gives:
\begin{equation} \label{j,k} \displaystyle a_k + \sum_{l \in S}^{} a_{kl} +a_{jk} =0. \end{equation}

\noindent Thus $a_{ik}=a_{jk}$ (from equations (\ref{i,k}) and (\ref{j,k})).

 Since $i, j, k$ are arbitrary, $a_{ik}=a_{jk}=a_{jl}$ for any $l \neq i, j, k$. So all the coefficients of $x_i x_j$ (for $i \neq j$) in $h$ are equal, say $a_{ij}=u; i \neq j$.

It follows from equation (\ref{i,k}) that $a_k +mu=0$. So \ $a_k = - mu \ \forall \ k$, which gives:
\begin{equation} \displaystyle h(x_1, \ldots, x_n)= u \left(-m \sum_{i=1}^{n} x_i^2 + \sum_{i < j}^{} x_i x_j \right). \nonumber \end{equation}

\noindent But then \ $h(\underbrace{1,\ldots, 1}_{m}, 0, \ldots, 0)=0$ \ gives \ 
$\displaystyle u \left(-m(m) + \frac{m(m-1)}{2} \right)=0,$
which implies $u=0$, which implies $h=0$.
\end{proof}

\begin{thm} \label{ODDvarPSDnotSOSquartic}
If $n \geq 5$ is odd, then $L_n$ is not sos.
\end{thm}

\begin{proof} Fix odd $n\geq 5, n=2m+1$. Then 

\vspace{-0.3cm}

$$L_{2m+1}= \displaystyle m(m+1)\sum_{i<j}^{}(x_i-x_j)^4 - \Big(\sum_{i<j}^{}(x_i-x_j)^2\Big)^2.$$

\noindent  If $L_{2m+1}=\displaystyle \sum_{t} h_t^2$, then $L_{2m+1}(\underline{x})=0 \Rightarrow$ each $h_t(\underline{x})=0$, for any $\underline{x} \in \mathbb{R}^n$.

In particular,  $L_{2m+1}(\underline{x})=0$ when $\underline{x}$ has $m$ or $(m+1)$ $1$'s and $(m+1)$ or $m$ $0$'s. 
So, $h_t(\underline{x})=0$ for $\underline{x}$ with $m$ or $(m+1)$ $1$'s and $(m+1)$ or $m$ $0$'s respectively. Write 
\[
h_t(\underline{x})=\displaystyle \sum_{i=1}^{n} a_i x_i^2 + \sum_{i <  j}^{n} a_{ij} x_i x_j.
\]
Then by Lemma \ref{QUAD h identically 0}, we get $h_t =0$. Hence $L_{2m+1}$ is not sos.
\end{proof}

Now we construct $f \in S\mathcal{P}_{2m,4} \setminus S\Sigma_{2m,4}$ for $m\geq 2$. 
\vspace{0.1cm}

Unfortunately, 
\[
L_{2m}(\underline{x})=\displaystyle \sum_{i<j}^{}(x_i-x_j)^2\Big(-(x_1+\ldots+x_{2m})+m(x_i+x_j)\Big)^2
\] 
(see \cite[Proposition 3.13]{Goel}) is sos, and so we need a different example in $S\mathcal{P}_{2m,4} \setminus S\Sigma_{2m,4}$. For  $2m \geq 4$, let 
\[
C_{2m}(x_1, \ldots, x_{2m}):= L_{2m+1}(x_1, \ldots, x_{2m}, 0).
\] 

\noindent Trivially, $C_{2m}$ is a symmetric $2m$-ary quartic and psd. We shall show it is not sos. 

\begin{thm} \label{EVENvarPSDnotSOSquartic}   For $m \geq 2$, $C_{2m}(x_1, \ldots, x_{2m})$ is not sos.
\end{thm}
\begin{proof} If $C_{2m}=\displaystyle \sum_{t} h_t^2$, then $C_{2m}(\underline{x})=0 \Rightarrow$ each $h_t(\underline{x})=0$, for any $\underline{x} \in \mathbb{R}^n$.

\noindent In particular,  $C_{2m}(\underline{x})=0$ when $\underline{x}$ has $m$ or $(m+1)$ $1$'s and $m$ or $(m-1)$ $0$'s. 
So, $h_t(\underline{x})=0$ for $\underline{x}$ with $m$ or $(m+1)$ $1$'s and $m$ or $(m-1)$ $0$'s respectively.

Write 
\[h_t(\underline{x})=\displaystyle \sum_{i=1}^{n} a_i x_i^2 + \sum_{i <  j}^{n} a_{ij} x_i x_j.
\]
Then by Lemma \ref{QUAD h identically 0}, we get $h_t =0$. Hence, $C_{2m}$ is not sos.
\end{proof}
 
To sum up, the answer to $\mathcal{Q}(S)$ can be summarised by  the same chart as for Hilbert's Theorem, given in the Introduction.

\section*{Acknowledgements}
This paper contains results from the Ph.D. Thesis \cite{Goel}. Charu Goel is thankful to MFO for inviting her as Leibnitz-Graduate-Student in April 2014. 
 Bruce Reznick wishes to thank the Simons Foundation: A Simons Visiting Professorship
supported his visit to the University of Konstanz in April 2014 and a Simons Collaboration Grant
supported his participation at MFO.

\end{document}